\documentclass[review]{elsarticle}
\usepackage{lineno}
\usepackage[latin1]{inputenc}
\usepackage{times}
\usepackage{amssymb,mathrsfs, amsmath}
\usepackage{amsmath,amsthm}
\usepackage{amssymb}
\usepackage{latexsym}
\usepackage{amsfonts}
\usepackage{mathrsfs}
\usepackage{epsfig}
\usepackage{hyperref}
\usepackage{graphicx}
\usepackage{graphics}
\usepackage{fancyhdr}
\newtheorem{thm}{Theorem}[section]

\newtheorem{defn}{Definition}[section]
\newtheorem{prop}{Proposition}[section]

\newtheorem{rem}{Remark}[section]

\newtheorem{cor}{Corollary}[section]

\newtheorem{exmpl}{Example}[section]

\begin{document}

\begin{frontmatter}

\title{Cohomology and deformation  of module  homomorphisms}
\author{ RB Yadav $^1$\fnref{}\corref{mycorrespondingauthor}}
\ead{rbyadav15@gmail.com}
\author{ Liangyun Chen$^2$, Yao Ma$^2$ }
\ead{ chenly640@nenu.edu.cn, may703@nenu.edu.cn}
\author{  Ying Hou$^3$}
\ead{chenly640@nenu.edu.cn}
\address{$^1$ Department of Mathematics, Sikkim University, Gangtok, Sikkim, 737102, \textsc{India}\\$^2$ School of Mathematics and Statistics, Northeast Normal University, Changchun, 130024, \textsc{China}\\$^3$Department of Mathematics, Northeast Forestry University, Harbin, 150040,\textsc{China}}
\cortext[mycorrespondingauthor]{Corresponding author}
\begin{abstract}
In this paper, we mainly focus on formal deformation theory of module homomorphisms. We first introduce the cohomology of module homomorphisms and study formal one-parameter  deformation. We obtain some properties about obstructions. Then we give some examples of deformations of modules and module homomorphisms.
\end{abstract}
\begin{keyword}
\texttt{Hochschild cohomology,  deformations, modules}
\MSC[2020] 13D03 \sep 13D10\sep 14D15 \sep 	16E40
\end{keyword}

\end{frontmatter}
\section{Introduction}\label{rbsec1}
 M. Gerstenhaber introduced algebraic deformation theory in a series of papers \cite{MG1},\cite{MG2},\cite{MG3}, \cite{MG4}, \cite{MG5}. He studied deformation theory of associative algebras.  Deformation theory of associative algebra morphisms was studied by M. Gerstenhaber and S.D. Schack \cite{GS1}, \cite{GS2}, \cite{GS3}. Deformation theory of Lie algebras was studied by Nijenhuis and Richardson \cite{NR1}, \cite{NR2}. Algebraic deformations of modules were first studied  by Donald and Flanigan \cite{DF}. They had to restrict themselves to finite dimensional algebras R over a field k and finite dimensional R-modules M. Recently, deformation theory of modules (without any restriction on dimension) was studied in \cite{DY1}.

The above representative and significant works inspired us to work on cohomology and deformation theory  of module homomorphisms.

Organization of the paper is as follows. In Section \ref{rbsec2}, we recall some definitions and results about deformation of module. In Section \ref{rbsec3}, we introduce  deformation complex and  deformation cohomology of a module homomorphism. In Section \ref{rbsec4}, we introduce  deformation of a module  homomorphism. In this section  we prove one of our most important  results that obstructions to deformations are cocycles. In Section \ref{rbsec5}, we study equivalence of two deformations  of a module  homomorphism. In Section \ref{chm6}, we give some examples of deformations of modules and module homomorphisms.  We show that if $A=k,$ then every deformation of  the module $M$ is trivial, that is $M$ is rigid. Using this we give large class of examples of deformations of a module homomorphism $\phi:M\to N.$

\section{Preliminaries}\label{rbsec2}
In this section, we recall definition of  Hochschild cohomology,  and  deformation  of a module from \cite{DY1}.  Throughout this paper, $k$ denotes a  commutative ring with unity, $A$ denotes an associative $k$-algebra, and $M$ denotes a  (left) $A$-module. Also, we write $\otimes$ for $\otimes_k$, the tensor product  over $k$, and $A^{\otimes n}$ for $A\otimes\cdots\otimes A$ (n factors). We use notation $(x,y)$ for both $x\oplus y\in M_1\oplus M_2$ and $x\otimes y\in A^{\otimes2}$ and recognize them from  context.
  Let $A$ be an associative $k$-algebra and $F$ be an $A$-bimodule .  Let $C^n(A;F)=Hom_k(A^{\otimes n},F),$ for all integers $n\ge 0.$ In particular, $C^0(A;F)=Hom_k(k,F)\equiv F.$ Also, define a $k$-linear map $\delta^n:C^n(A;F)\to C^{n+1}(A;F)$ given by
  \begin{eqnarray*}
    \delta ^nf(x_1,\cdots, x_{n+1}) &=& x_1f(x_2,\cdots,x_{n+1})+\sum_{i=1}^{n}(-1)^if(x_1,\cdots,x_ix_{i+1}, \cdots, x_{n+1}) \\
     && +(-1)^{n+1}f(x_1,\cdots,x_n)x_{n+1},
  \end{eqnarray*}
   for $n\ge 1.$ $\delta^0(m)(a)=am-ma,$ for all $m\in F$, $a\in A.$  This gives a cochain complex  $(C^{\ast}(A;F),\delta)$ ,  cohomology of which is denoted by $H^{\ast}(A;F)$ and called as Hochschild cohomology of $A$ with coefficients in $F$.

Let $M$ and $N$ be (left) $A$-modules. The set of $k$-linear maps from $M$ to $N$, $Hom_k(M,N)$, has a structure of  an $A$-bimodule  such that $$(rf)(m)=r(f(m))\;\; \text{and}\; (fs)(m)=f(sm),$$ for all $r,s\in A$, $f\in Hom_k(M,N)$ and $m\in M$.   In particular, the set of $k$-linear endomorphisms of $M$, $End(M)$ is  $A$-bimodule. Moreover,  $End(M)$ is also an associative $k$-algebra with composition of endomorphisms as product.

From \cite{DY1}, we recall definition of deformation of a left $A$-module $M$. Note that $A$-module structure on $M$ is equivalent to an associative algebra morphism $\xi:A\to End(M)$  such that $\xi(r)m=rm,$ for all $r\in A$ and $m\in M.$
\begin{defn}
Let $A$ be an associative $k$-algebra and $M$ be a left  $A$-module.
Define $C^n(M)=C^n(A,End(M))$, $\forall n\ge 0.$ Then $(C^{\ast}(M),\delta)$ is a cochain complex. We call the cohomology of this complex as deformation cohomology of $M$ and denote it by $H^{\ast}(M)$.\\
 A formal one-parameter deformation of $M$ is defined to be the formal power series $\xi_t=\sum_{i=0}^{\infty}\xi_it^i $ such that
  \begin{itemize}
    \item [(a)]  $\xi_i\in Hom_k(A,End(M))$, $\forall$ i, $\xi_0=\xi.$
    \item [(b)]$\xi_t(rs)=\xi_t(r)\xi_t(s),$ $\forall r,s\in A.$
   \end{itemize}
   \end{defn}
  \begin{rem}
  Note that condition (b) in above definition is equivalent to $\xi_n(rs)=\sum_{i+j=n}\xi_i(r)\xi_j(s),$ for all $n\ge 0.$
  \end{rem}
\begin{defn}
  A formal one-parameter deformation of order n for $ M $ is defined to be the formal power series $\xi_t=\sum_{i=0}^{n}\xi_it^i $ such that
  \begin{itemize}
    \item [(a)]  $\xi_i\in Hom_k(A,End(M))$, $\forall$ i, $\xi_0=\xi.$
    \item [(b)]$\xi_t(rs)=\xi_t(r)\xi_t(s),$  (modulo $t^{n+1}$) $\forall r,s\in A.$
   \end{itemize}
   \end{defn}
 \begin{rem}
  Note that condition (b) in above definition is equivalent to $\xi_l(rs)=\sum_{ i+j=l}\xi_i(r)\xi_j(s)$, for all $n\ge l\ge 0.$
 \end{rem}
\section{ Deformation complex of a module homomorphism}\label{rbsec3}
\begin{defn}\label{def3.1}
Let $M$, $N$ be left $A$-modules and $\phi:M\to N$ be an  $A$-module  homomorphism. We define
$$C^n(\phi)=C^n(A;End(M))\oplus C^n(A;End(N))\oplus C^{n-1}(A;Hom_k(M,N)),$$ for all $n\in \mathbb{N}$ and $C^0(\phi)=0$.
For any  $A$-module homomorphism $\phi:M\to N $, $u\in C^n(A;End(M))$, $v\in C^n(A;End(N))$, define    $\phi u:A^{\otimes n}\to Hom(M,N)$ and $v\phi:A^{\otimes n}\to Hom_k(M,N)$ by $\phi u(x_1, x_2,\cdots, x_n)(m)=\phi( u(x_1, x_2,\cdots, x_n)(m))$,   $v\phi(x_1, x_2,\cdots, x_n)(m)=v(x_1, x_2, \cdots,  x_n)(\phi (m)),$ for all $(x_1, x_2,\cdots, x_n)\in A^{\otimes n},$ $m\in M.$  Also, we define $d^n:C^n(\phi)\to C^{n+1}(\phi)$ by $$d^n(u,v,w)=(\delta^n u, \delta^n v, \phi u-v\phi -\delta^{n-1}w ),$$ for all $(u,v,w)\in C^n(\phi).$ Here the $\delta^n$'s denote coboundaries of the cochain complexes $C^{\ast}(A;End(M))$,  $C^{\ast}(A;End(N))$ and $C^{\ast}(A;Hom_k(M,N))$.
\end{defn}
\begin{prop}
  $(C^{\ast}(\phi), d)$ is a cochain complex.
\end{prop}
\begin{proof}
  We have \begin{eqnarray*}
           &&d^{n+1}d^n(u,v,w)\\
              &=& d^{n+1}(\delta^n u, \delta^n v, \phi u-v\phi -\delta^{n-1}w ) \\
             &=& (\delta^{n+1}\delta^n u, \delta^{n+1}\delta^n v, \phi(\delta^n u)- (\delta^n v)\phi- \delta^n( \phi u-v\phi -\delta^{n-1}w ))
          \end{eqnarray*} One can easily see that $\delta^n( \phi u-v\phi)=\phi (\delta^n u)- (\delta^n v)\phi$. So, since $\delta^{n+1}\delta^n u=0,$ $\delta^{n+1}\delta^n v=0$, $\delta^{n+1}\delta^n w=0$, we have $d^{n+1}d^n=0$. Hence we conclude the result.
\end{proof}

We call the cochain complex $(C^{\ast}(\phi),d)$  as deformation complex of $\phi,$ and the corresponding cohomology as  deformation cohomology of $\phi$. We denote the  deformation cohomology by $H^n(\phi)$, that is, $H^n(\phi)=H^n(C^{\ast}(\phi),d)$. Next proposition relates $H^{\ast}(\phi)$ to $H^{\ast}(A,End(M))$,  $H^{\ast}(A,End(N))$ and $H^{\ast}(A,Hom_k(M,N))$.

\begin{prop}\label{rb-99}
  \hspace{0.3cm}If\hspace{0.3cm}  $H^n(A,End(M))=0$,\hspace{0.3cm}  $H^n(A,End(N))=0$\hspace{0.3cm} and\hspace{0.3cm} \\$H^{n-1}(A,Hom_k(M,N))=0$, then $H^{n}(\phi)=0$.
\end{prop}
\begin{proof}
  Let $(u,v,w)\in C^n(\phi)$ be a cocycle, that is, $d^n(u,v,w)=(\delta^n u, \delta^n v, \phi u-v\phi -\delta^{n-1}w )=0$. This implies that $\delta^n u=0$, $\delta^n v=0$, $\phi u-v\phi -\delta^{n-1}w =0$. $H^n(A,End(M))=0 \Rightarrow u=\delta^{n-1}u_1$, for some $u_1\in C^{n-1}(A,End(M))$, and  $H^n(A,End(N))=0\Rightarrow \delta^{n-1}v_1=v$, for some $v_1\in C^{n-1}(A,End(N))$.
  So \begin{eqnarray*}
       0 &=& \phi u-v\phi -\delta^{n-1}w \\
        &=& \phi(\delta^{n-1}u_1)-(\delta^{n-1}v_1)\phi-\delta^{n-1}w\\
        &=& \delta^{n-1}(\phi u_1)-\delta^{n-1}(v_1\phi)-\delta^{n-1}w \\
        &=& \delta^{n-1}(\phi u_1-v_1\phi-w).
     \end{eqnarray*}
    Hence $\phi u_1-v_1\phi-w\in C^{n-1}(A,Hom_k(M,N))$ is a cocycle. Now,
     \begin{eqnarray*}
       H^{n-1}(A,Hom_k(M,N))=0 &\Rightarrow& \phi u_1-v_1\phi-w=\delta^{n-2}w_1,\\
       && \text{for some}\;  w_1\in C^{n-2}(A,Hom_k(M,N)), \\
        &\Rightarrow& \phi u_1-v_1\phi-\delta^{n-2}w_1=w.
     \end{eqnarray*}
      Thus $(u,v,w)=(\delta^{n-1} u_1, \delta^{n-1 }v_1, \phi u_1-v_1\phi -\delta^{n-2}w_1 )=d^{n-1}(u_1,v_1,w_1)$, for some $(u_1,v_1,w_1)\in C^{n-1}(\phi).$  Thus every cocycle in $C^n(\phi)$ is a coboundary. Hence we conclude that $H^{n}(\phi)=0$.
\end{proof}
\section{ Deformation of  a module  homomorphism}\label{rbsec4}
\begin{defn}\label{rb2}

Let $M$ and $N$ be (left) $A$-modules. A formal one-parameter deformation of a module  homomorphism $\phi:M\to N$ is a triple $(\xi_t, \eta_t, \phi_t)$, in which:

\begin{enumerate}
  \item  $\xi_t=\sum_{i=0}^{\infty}\xi_i t^i$ is a formal one-parameter deformation for $M$.
  \item  $\eta_t=\sum_{i=0}^{\infty}\eta_it^i$ is a formal one-parameter deformation for  $N$.
  \item  $\phi_t=\sum_{i=0}^{\infty}\phi_it^i$, where $\phi_i:M\to N$ is  a module  homomorphism such that $\phi_t(\xi_t(r)m)=\eta_t(r)\phi_t(m),$  for all $r\in A$, $m\in M$ and $\phi_0=\phi$.
\end{enumerate}
\end{defn}
\begin{rem}
  Note that a triple  $(\xi_t, \eta_t, \phi_t)$, as given above, is a formal one-parameter deformation of $\phi$ provided following properties are satisfied.
\begin{itemize}
  \item[(i)] $\xi_t(rs)=\xi_t(r)\xi_t(s)$, for all $r,s\in A;$
   \item[(ii)]$\eta_t(rs)=\eta_t(r)\eta_t(s)$, for all $r,s\in A;$
  \item[(iii)] $\phi_t(\xi_t(r)m)=\eta_t(r)\phi_t(m),$  for all $r\in A$, $m\in M$ .
\end{itemize}
The conditions $(i)$, $(ii)$ and $(iii)$ are equivalent to following conditions respectively.
\begin{equation}\label{rbeqn1}
  \xi_l(rs)=\sum_{i+j=l}\xi_i(r)\xi(s), \;\text{for all} \;r,s\in A,\; l\ge 0.
  \end{equation}
  \begin{equation}\label{rbeqn2}
   \eta_l(rs)=\sum_{i+j=l}\eta_i(r)\eta(s), \;\text{for all}\; r,s\in A,\; l\ge 0.
   \end{equation}
   \begin{equation}\label{rbeqn3}
      \sum_{i+j=l}\phi_i(\xi_j(r)m)=\sum_{i+j=l}\eta_i(r)(\phi_j(m)); \;\text{for all}\; r\in A,\; m\in M\;  l\ge 0.
\end{equation}
\end{rem}
 Now we define  deformation  of finite order.

 \begin{defn}\label{rb3}
Let $M$ and $N$ be (left) $A$-module. A deformation of order n  of a module homomorphism $\phi:A\to B$ is a triple $(\xi_t, \eta_t, \phi_t)$, in which:

\begin{enumerate}
   \item  $\xi_t=\sum_{i=0}^{n}\xi_i t^i$ is a formal one-parameter deformation of order n for $M$.
  \item  $\eta_t=\sum_{i=0}^{n}\eta_it^i$ is a formal one-parameter deformation of order n  for  $N$.
  \item  $\phi_t=\sum_{i=0}^{n}\phi_it^i$, where $\phi_i:M\to N$ is  a module  homomorphism such that $\phi_t(\xi_t(r)m)=\eta_t(r)\phi_t(m),$ (modulo $t^{n+1}$) for all $r\in A$, $m\in M$ and $\phi_0=\phi$.
\end{enumerate}

\end{defn}
\begin{rem}\label{rbrem1}
  \begin{itemize}
    \item For $l=0$, conditions \ref{rbeqn1}, \ref{rbeqn2} and \ref{rbeqn3} are equivalent to the fact that $M$ and $N$ are  (left) $A$-modules and $\phi$ is a module homomorphism, respectively.
    \item For $l=1$, conditions \ref{rbeqn1}, \ref{rbeqn2} and \ref{rbeqn3} are equivalent to $\delta^1\xi_1=0,$ $\delta^1\eta_1=0$ and $\phi\xi_1-\eta_1\phi -\delta \phi_1=0,$ respectively.
          Thus for $l=1$, conditions  \ref{rbeqn1}, \ref{rbeqn2} and \ref{rbeqn3} are equivalent to saying that $(\xi_1,\eta_1,\phi_1)\in C^1(\phi)$ is a cocycle. In general, for $l\ge 2$, $(\xi_l,\eta_l,\phi_l)$ is just a 1-cochain in $C^1(\phi).$
          \item Condition (3) in Definition  \ref{rb3} is equivalent to $$\sum_{i+j=l}\phi_i(\xi_j(r)m)=\sum_{i+j=l}\eta_i(r)(\phi_j(m)); \;\text{for all}\; r\in A,\; m\in M\;  n\ge l\ge 0$$.
  \end{itemize}
\end{rem}
\begin{defn}
  The 1-cochain  $(\xi_1, \eta_1, \phi_1)$ in $C^1(\phi)$ is called infinitesimal of the   deformation $(\xi_t, \eta_t, \phi_t)$. In general, if $(\xi_i,\eta_i,\phi_i)=0,$ for $1\le i\le n-1$, and $(\xi_n,\eta_n,\phi_n)$ is a nonzero cochain in  $C^1(\phi,\phi)$, then $(\xi_n,\eta_n,\phi_n)$ is called n-infinitesimal of the deformation $(\xi_t, \eta_t, \phi_t)$.
\end{defn}
\begin{prop}
  The infinitesimal   $(\xi_1, \eta_1, \phi_1)$ of the  deformation  $(\xi_t, \eta_t, \phi_t)$ is a 1-cocycle in $C^1(\phi).$ In general, n-infinitesimal  $(\xi_n, \eta_n, \phi_n)$ is a cocycle in $C^1(\phi).$
\end{prop}
\begin{proof}
  For n=1, proof is obvious from the Remark \ref{rbrem1}. For $n>1$, proof is similar.
\end{proof}
We can write Equations \ref{rbeqn1}, \ref{rbeqn2} and \ref{rbeqn3} for $l=n+1$ using the definition of coboundary $\delta$ as
\begin{equation}\label{rbeqn4}
  \delta^1\xi_{n+1}(a,b)=-\sum_{\substack{ i+j=n+1\\i,j>0}}\xi_i(a)\xi_j(b), \;\text{for all}\; a,b\in A.
  \end{equation}
  \begin{equation}\label{rbeqn5}
    \delta^1\eta_{n+1}(a,b)=-\sum_{\substack{ i+j=n+1\\i,j>0}}\eta_i(a)\eta_j(b), \;\text{for all}\; a,b\in A.
   \end{equation}
   \begin{eqnarray}\label{rbeqn6}
     (\phi\xi_{n+1})(a)&-&(\eta_{n+1}\phi)(a) -\delta^0\phi_{n+1}(a) \notag\\
      &=& \sum_{\substack{ i+j=n+1\\i,j>0}}(\eta_i\phi_j)(a) -\sum_{\substack{ i+j=n+1\\i,j>0}}(\phi_i\xi_j)(a) ,
   \end{eqnarray}
 for all  $a\in A$ .
By using   Equations \ref{rbeqn4}, \ref{rbeqn5} and \ref{rbeqn6}   we have
\begin{align}
&d^1(\xi_{n+1},\eta_{n+1},\phi_{n+1})(a,b,x,y,p) \notag\\
&= ( -\sum_{\substack{ i+j=n+1\\i,j>0}}\xi_i(a)\xi_j(b),
    -\sum_{\substack{ i+j=n+1\\i,j>0}}\eta_i(x)\eta_j(y),\notag\\
   & \hspace{3cm}\sum_{\substack{ i+j=n+1\\i,j>0}}(\eta_i\phi_j)(p) -\sum_{\substack{ i+j=n+1\\i,j>0}}(\phi_i\xi_j)(p) ),
   \end{align}
for all $a,b,x,y,p\in A$.\\
Define a 2-cochain $F_{n+1}$ by
\begin{align}
&F_{n+1}(a,b,x,y,p)\notag\\
&= ( -\sum_{\substack{ i+j=n+1\\i,j>0}}\xi_i(a)\xi_j(b),
    -\sum_{\substack{ i+j=n+1\\i,j>0}}\eta_i(x)\eta_j(y),\notag\\
   & \hspace{3cm}\sum_{\substack{ i+j=n+1\\i,j>0}}(\eta_i\phi_j)(p) -\sum_{\substack{ i+j=n+1\\i,j>0}}(\phi_i\xi_j)(p) ).
\end{align}

\begin{defn}
  The 2-cochain $F_{n+1}\in C^2(\phi)$ is called $(n+1)th$ obstruction cochain for extending the given  deformation of order n to a deformation of $\phi$ of order $(n+1)$. Now onwards we denote $F_{n+1}$ by $Ob_{n+1}(\phi_t)$.
\end{defn}
We have the following result.
\begin{thm}\label{thm4.1}
  The (n+1)th obstruction cochain $Ob_{n+1}(\phi_t)$ is a 2-cocycle.
\end{thm}
\begin{proof}
  We have,
  $$d^2 Ob_{n+1}=(\delta^2(O_1),\delta^2(O_2),\phi O_1-O_2\phi-\delta^1(O_3)),$$
  where $O_1$, $O_2$ and $O_3$ are given by  $$O_1(a,b)=  -\sum_{\substack{ i+j=n+1\\i,j>0}}\xi_i(a)\xi_j(b),$$
  $$O_2(x,y)= -\sum_{\substack{ i+j=n+1\\i,j>0}}\eta_i(x)\eta_j(y),$$
  $$O_3(p)= \sum_{\substack{ i+j=n+1\\i,j>0}}(\eta_i\phi_j)(p) -\sum_{\substack{ i+j=n+1\\i,j>0}}(\phi_i\xi_j)(p) .$$
  From \cite{DY1}, we have $\delta^2(O_1)=0$,  $\delta^2(O_2)=0$. So, to prove that $d^2Ob_{n+1}=0$, it remains to show that $\phi O_1-O_2\phi-\delta^1(O_3)=0.$
  To prove that $\phi O_1-O_2\phi-\delta^1(O_3)=0$, we use similar ideas as have been used in \cite{AM} and \cite{DY}.
  We have,
  \begin{align}\label{rbeqn7}
   (\phi O_1 -O_2\phi )(x,y)=-\sum_{\substack{ i+j=n+1\\i,j>0}}\phi\xi_i(x)\xi_j(y) +\sum_{\substack{ i+j=n+1\\i,j>0}}\eta_i(x)\eta_j(y)\phi
 \end{align}
  and \begin{eqnarray}\label{rbeqn8}
        \delta^1(O_3)(x,y) &= \sum_{\substack{ i+j=n+1\\i,j>0}}\eta_0(x)(\eta_i\phi_j)(y)-\sum_{\substack{ i+j=n+1\\i,j>0}}(\eta_i\phi_j)(xy)\notag\\ & +\sum_{\substack{ i+j=n+1\\i,j>0}}(\eta_i\phi_j)(x)\xi_0(y)-\sum_{\substack{ i+j=n+1\\i,j>0}}\eta_0(x)(\phi_i\xi_j)(y)\notag\\
       & +\sum_{\substack{ i+j=n+1\\i,j>0}}(\phi_i\xi_j)(xy)-\sum_{\substack{ i+j=n+1\\i,j>0}}(\phi_i\xi_j)(x)\xi_0(y).
\end{eqnarray}
From Equation \ref{rbeqn3}, we have
\begin{eqnarray}\label{rbeqn9}
   \phi_j\xi_0(y)&=& \sum_{\substack{ \alpha+\beta=j\\\alpha,\beta\ge 0}}\eta_{\alpha}(y)\phi_{\beta}-\sum_{\substack{ p+q=j\\1\le q\le j}}\phi_p\xi_q(y).
\end{eqnarray}
Substituting expression for $\phi_j\xi_0$ from Equation \ref{rbeqn9}, in the third sum on the right hand side of Equation \ref{rbeqn8} we can rewrite it as
\begin{eqnarray}\label{rbeqn10}
 \sum_{\substack{ i+j=n+1\\i,j>0}}(\eta_i\phi_j)(x)\xi_0(y)  &=&\sum_{\substack{ i+j=n+1\\i,j>0}}\sum_{\substack{ \alpha+\beta=j\\\alpha,\beta\ge 0}}\eta_i(x)\eta_{\alpha}(y)\phi_{\beta}\notag\\
 &&-\sum_{\substack{ i+j=n+1\\i,j>0}}\sum_{\substack{ p+q=j\\1\le q\le j}}\eta_i(x)\phi_p\xi_q(y).
\end{eqnarray}

         The first sum of Equation \ref{rbeqn10} splits into two sums as
         \begin{align}\label{rbeqn13}
        \sum_{\substack{ i+j=n+1\\i,j>0}}\sum_{\substack{ \alpha+\beta=j\\\alpha,\beta\ge 0}}\eta_i(x)\eta_{\alpha}(y)\phi_{\beta}=\sum_{\substack{ i+j=n+1\\i,j>0}}\sum_{\substack{ \alpha+\beta=j\\\beta> 0}}\eta_i(x)\eta_{\alpha}(y)\phi_{\beta}+ \sum_{\substack{ i+j=n+1\\i,j>0}}\eta_i(x)\eta_j(y)\phi\notag.\\
         \end{align}
         The second sum on the r.h.s. of  Equation \ref{rbeqn13} appears as second  sum on the  r.h.s. of Equation \ref{rbeqn7}.
         By applying a similar argument to the fourth sum on the r.h.s. of Equation \ref{rbeqn8}, using Equation \ref{rbeqn3} on $\phi_k\mu_0(y,z)$, one can rewrite it as
         \begin{eqnarray}\label{rbeqn14}
  -\sum_{\substack{ i+j=n+1\\i,j>0}}\eta_0(x)(\phi_i\xi_j)(y)&=& \sum_{\substack{ i+j=n+1\\i,j>0}}\sum_{\substack{ \alpha+\beta=i\\1\le \alpha\le i 0}}\eta_{\alpha}(x)\phi_{\beta}\xi_j(y)\notag\\
 &&-\sum_{\substack{ i+j=n+1\\i,j>0}}\sum_{\substack{ p+q=i\\p, q\ge  0}}\phi_p\xi_q(x)\xi_j(y).
\end{eqnarray}
The second sum of Equation   \ref{rbeqn14} splits into two sums as
\begin{eqnarray}\label{rbeqn12}
 -\sum_{\substack{ i+j=n+1\\i,j>0}}\sum_{\substack{ p+q=i\\p, q\ge  0}}\phi_p\xi_q(x)\xi_j(y)  &=&  -\sum_{\substack{ i+j=n+1\\i,j>0}}\sum_{\substack{ p+q=i\\p>  0}}\phi_p\xi_q(x)\xi_j(y) \\
 && -\sum_{\substack{ i+j=n+1\\i,j>0}}\phi\xi_i(x)\xi_j(y).
 \end{eqnarray}

As above second sum on r.h.s. of Equation  \ref{rbeqn12} is first sum on the r.h.s. of Equation \ref{rbeqn7}.

In the first sum on the r.h.s. of Equation \ref{rbeqn8}, we use Equation \ref{rbeqn2} to substitute $\eta_0(x)\eta_i(y)$ to obtain
\begin{eqnarray}\label{rbeqn15}
\sum_{\substack{ i+j=n+1\\i,j>0}}\eta_0(x)\eta_i(y)\phi_j&=&\sum_{\substack{ i+j=n+1\\i,j>0}}\eta_i(xy)\phi_j\notag\\
 &&-\sum_{\substack{ i+j=n+1\\i,j>0}}\sum_{\substack{ \alpha+\beta=i\\1\le \alpha\le i 0}}\eta_{\alpha}(x)\eta_{\beta}(y)\phi_j.
 \end{eqnarray}
First sum on the r.h.s. of Equation \ref{rbeqn15} cancels with the second sum on the r.h.s. of Equation \ref{rbeqn8}.
In the sixth sum on the r.h.s. of Equation \ref{rbeqn8}, we use Equation \ref{rbeqn1} to substitute $\xi_j(x)\xi_0(y)$ to obtain
\begin{eqnarray}\label{rbeqn16}
-\sum_{\substack{ i+j=n+1\\i,j>0}}\phi_i\xi_j(x)\xi_0(y)&=& \sum_{\substack{ i+j=n+1\\i,j>0}}\sum_{\substack{ \alpha+\beta=j\\1\le \beta\le j 0}}\phi_i\xi_{\alpha}(x)\xi_{\beta}(y)\notag.\\
 &&-\sum_{\substack{ i+j=n+1\\i,j>0}}\phi_i\xi_j(xy).
\end{eqnarray}

The second  sum on the r.h.s. of Equation \ref{rbeqn16} cancels with the fifth sum on the r.h.s. of Equation \ref{rbeqn8}.

 From our previous arguments we have,
\begin{eqnarray}\label{rbeqn20}
  & &\phi O_1-O_2\phi-\delta^2(O_3)(x,y)\nonumber\\
   &=& -\sum_{\substack{ i+j=n+1\\i,j>0}}\sum_{\substack{ p+q=j\\1\le q\le j}}\eta_i(x)\phi_p\xi_q(y)+ \sum_{\substack{ i+j=n+1\\i,j>0}}\sum_{\substack{ \alpha+\beta=i\\1\le \alpha\le i }}\eta_{\alpha}(x)\phi_{\beta}\xi_j(y)\notag \\
   &&-\sum_{\substack{ i+j=n+1\\i,j>0}}\sum_{\substack{ p+q=i\\p>  0}}\phi_p\xi_q(x)\xi_j(y)+ \sum_{\substack{ i+j=n+1\\i,j>0}}\sum_{\substack{ \alpha+\beta=j\\1\le \beta\le j 0}}\phi_i\xi_{\alpha}(x)\xi_{\beta}(y)\notag\\
   &&-\sum_{\substack{ i+j=n+1\\i,j>0}}\sum_{\substack{ \alpha+\beta=i\\1\le \alpha\le i 0}}\eta_{\alpha}(x)\eta_{\beta}(y)\phi_j+\sum_{\substack{ i+j=n+1\\i,j>0}}\sum_{\substack{ \alpha+\beta=j\\\beta> 0}}\eta_i(x)\eta_{\alpha}(y)\phi_{\beta}.
    \end{eqnarray}
      Moreover, we have \begin{eqnarray}\label{rbeqn21}
                \sum_{\substack{ i+j=n+1\\i,j>0}}\sum_{\substack{ p+q=j\\1\le q\le j}}\eta_i(x)\phi_p\xi_q(y) &=& \sum_{\substack{ i+j=n+1\\i,j>0}}\sum_{\substack{ \alpha+\beta=i\\1\le \alpha\le i }}\eta_{\alpha}(x)\phi_{\beta}\xi_j(y)\notag \\
                &=&\sum_{\substack{\alpha+\beta+\gamma=n+1\\\alpha,\gamma>0\\\beta\ge 0}}\eta_{\alpha}(x)\phi_{\beta}\xi_{\gamma}(y).
              \end{eqnarray}
              \begin{eqnarray}\label{rbeqn22}
                \sum_{\substack{ i+j=n+1\\i,j>0}}\sum_{\substack{ p+q=j\\1\le q\le j}} \phi_p\xi_q(x)\xi_j(y)&=& \sum_{\substack{ i+j=n+1\\i,j>0}}\sum_{\substack{ \alpha+\beta=i\\1\le \alpha\le i }}\phi_i\xi_{\alpha}(x)\xi_{\beta}(y)\notag \\
                &=&\sum_{\substack{\alpha+\beta+\gamma=n+1\\\alpha,\gamma>0\\\beta\ge 0}}\phi_{\alpha}\xi_{\beta}(x)\xi_{\gamma}(y).
              \end{eqnarray}
              \begin{eqnarray}\label{rbeqn23}
                \sum_{\substack{ i+j=n+1\\i,j>0}}\sum_{\substack{ \alpha+\beta=j\\1\le \beta\le j}} \eta_{\alpha}(x)\eta_{\beta}(y)\phi_j&=& \sum_{\substack{ i+j=n+1\\i,j>0}}\sum_{\substack{ \alpha+\beta=i\\1\le \alpha\le i }}\eta_i(x)\eta_{\alpha}(y)\phi_{\beta}\notag \\
                &=&\sum_{\substack{\alpha+\beta+\gamma=n+1\\\alpha,\gamma>0\\\beta\ge 0}}\eta_{\alpha}(x)\eta_{\beta}(y)\phi_{\gamma}.
              \end{eqnarray}
 Hence, from Equations \ref{rbeqn20}, \ref{rbeqn21}, \ref{rbeqn22} and  \ref{rbeqn23},  we have $$\phi O_1-O_2\phi-\delta^2(O_3)(x,y)=0.$$

  This completes the proof of the theorem.
\end{proof}
\begin{thm}\label{thm4.2}
Let $(\xi_t,\eta_t,\phi_t)$ be a deformation of $\phi$ of order n. Then $(\xi_t,\eta_t,\phi_t)$ extends to a deformation of order $n+1$ if and only if cohomology class of $(n+1)$th obstruction $Ob_{n+1}(\phi_t)$  vanishes.
\end{thm}
\begin{proof}
  Suppose that a deformation $(\xi_t,\eta_t,\phi_t)$ of $\phi$ of order n extends to a deformation of order $n+1$. This implies that Equations  \ref{rbeqn1},\ref{rbeqn2} and \ref{rbeqn3} are satisfied for $r=n+1.$  Observe that this implies $Ob_{n+1}(\phi_t)=d^1(\xi_{n+1},\eta_{n+1},\phi_{n+1})$. So cohomology class of  $Ob_{n+1}(\phi_t)$ vanishes. Conversely, suppose that  cohomology class of  $Ob_{n+1}(\phi_t)$ vanishes, that is, $Ob_{n+1}(\phi_t)$ is a coboundary. Let
  $$ Ob_{n+1}(\phi_t)=d^1(\xi_{n+1},\eta_{n+1},\phi_{n+1}),$$
  for some 1-cochain  $(\xi_{n+1},\eta_{n+1},\phi_{n+1})\in C^1(\phi).$ Take
  $$(\tilde{\xi_t},\tilde{\eta_t},\tilde{\phi_t})=(\xi_t+\xi_{n+1}t^{n+1},\eta_t+\eta_{n+1}t^{n+1},\phi_t+\phi_{n+1}t^{n+1})$$.
  Observe that $(\tilde{\xi_t},\tilde{\eta_t},\tilde{\phi_t})$ satisfies Equations  \ref{rbeqn1},\ref{rbeqn2} and \ref{rbeqn3} for $0\le l\le n+1$. So  deformation  $(\tilde{\xi_t},\tilde{\eta_t},\tilde{\phi_t})$ of $\phi $ is an  extension of $(\xi_t,\eta_t,\phi_t)$ and its order is   $n+1$.

\end{proof}
\begin{cor}
  If $H^2(\phi)=0$, then every 1-cocycle in $C^1(\phi)$ is an infinitesimal of some formal deformation of $\phi.$
\end{cor}

\section{Equivalence of  deformations}\label{rbsec5}
 Recall from \cite{DY1} that a formal isomorphism between the  deformations $\xi_t$ and $\tilde{\xi_t}$ of a module  $M$ is a $k[[t]]$-linear automorphism  $\Psi_t:M[[t]]\to M[[t]]$ of the  form  $\Psi_t=\sum_{i\ge 0}\psi_it^i$, where each $\psi_i$ is a $k$-linear map $M\to M$, $\psi_0(a)=a$, for all $a\in A$ and $\tilde{\xi_t}(r)\Psi_t(m)=\Psi_t(\xi_t(r)m),$ for all $r\in A$, $m\in M$
\begin{defn}
 Let   $(\xi_t,\eta_t,\phi_t)$  and $(\tilde{\xi_t},\tilde{\eta_t},\tilde{\phi_t})$ be two  deformations of $\phi$.  A  formal isomorphism from $(\xi_t,\eta_t,\phi_t)$  to  $(\tilde{\xi_t},\tilde{\eta_t},\tilde{\phi_t})$ is a pair  $(\Psi_t,\Theta_t)$, where $\Psi_t:M[[t]]\to M[[t]]$ and $\Theta_t:N[[t]]\to N[[t]]$ are formal isomorphisms from $\xi_t$ to $\tilde{\xi_t}$ and $\eta_t$ to $\tilde{\eta_t}$, respectively,  such that $$\tilde{\phi_t}\circ\Psi_t=\Theta_t\circ\phi_t.$$
  Two formal  deformations $(\xi_t,\eta_t,\phi_t)$  and $(\tilde{\xi_t},\tilde{\eta_t},\tilde{\phi_t})$ are said to be equivalent if there exists a formal isomorphism  $(\Psi_t,\Theta_t)$ from $(\xi_t,\eta_t,\phi_t)$ to  $(\tilde{\xi_t},\tilde{\eta_t},\tilde{\phi_t})$.
\end{defn}
\begin{defn}
  Any  deformation of $\phi:M\to N$ that is equivalent to the deformation $(\xi_0,\eta_0,\phi)$ is said to be a trivial deformation.
\end{defn}
\begin{thm}
  The cohomology class of the infinitesimal of a  deformation $(\xi_t,\eta_t,\phi_t)$ of $\phi:A\to B$ is determined by the equivalence class of $(\xi_t,\eta_t,\phi_t)$.
\end{thm}
\begin{proof}
  Let  $(\Psi_t,\Theta_t)$ from  $(\xi_t,\eta_t,\phi_t)$ to  $(\tilde{\xi_t},\tilde{\eta_t},\tilde{\phi_t})$ be a  formal  isomorphism. So, we have  $\tilde{\xi_t}\Psi_t=\Psi_t\xi_t,$ $\tilde{\eta_t}\Theta_t=\Theta_t \eta_t,$ and   $\tilde{\phi_t}\circ\Psi_t=\Theta_t\circ\phi_t.$ This implies that $\xi_1-\tilde{\xi_1}=\delta^0\psi_1$, $\eta_1-\tilde{\eta_1}=\delta^0\theta_1$ and $\phi_1-\tilde{\phi_1}=\phi\psi_1-\theta_1\phi$. So we have $d^1(\psi_1,\theta_1,0)=(\xi_1,\eta_1,\phi_1)-(\tilde{\xi_1},\tilde{\eta_1},\tilde{\phi_1}).$ This finishes the proof.
\end{proof}

\section{Some Examples}\label{chm6}
In this section we give  examples of deformations of modules and module homomorphisms. 
\begin{exmpl}\label{mode1}
Let $k$ be a field. Take $A=k$. Let $M$ be a vector space over $k$. Then $M$ is a module over the associative algebra $A=k.$  Let $\xi_t=\sum_{i=0}^{\infty}\xi_it^i $ be a deformation of $M.$ By definition $\xi_i\in Hom_k(k,End(M)),$ $\forall i\ge 0$ and  $\xi_0(r)m=\xi(r)m=rm,$ $\forall r\in k,\; m\in M.$  Also, by definition we have, $\forall r, s\in k,$ 

$ \xi_t(rs)=\xi_t(r)\xi_t(s)$, that is,
\begin{eqnarray}\label{mod1}
\sum_{i=0}^{\infty}\xi_i(rs)t^i 
&=&\sum_{i=0}^{\infty}\xi_i(r)t^i \sum_{j=0}^{\infty}\xi_j(s)t^j\notag\\
&=& \sum_{l=0}^{\infty} \sum_{\substack{ i,j\ge 0,\\i+j=l}}\xi_i(r)\xi_j(s) t^l 
\end{eqnarray}
 Since every $\xi_i(r)=r\xi_i(1)$, $\forall \xi_i\in Hom_k(k,End(M)),$ using Equation \ref{mod1}, we have 
 \begin{eqnarray}\label{mod2}
 rs\sum_{l=0}^{\infty}\xi_l(1)t^l &=rs&\sum_{l=0}^{\infty} \sum_{\substack{ i,j\ge 0,\\i+j=l}}\xi_i(1)\xi_j(1) t^l 
 \end{eqnarray}
 From Equation \ref{mod2}, we have 
 \begin{equation}\label{mod3}
 \xi_l(1)-\sum_{\substack{ i,j\ge 0,\\i+j=l}}\xi_i(1)\xi_j(1)=0,\; \forall l\ge 0.
 \end{equation}
  From Equation \ref{mod3}, we have 
 \begin{enumerate}
 \item For $l=1$,  $\xi_1(1)-2\xi_1(1)=0$, that is $\xi_1=0.$
 \item For $l=2$,  $\xi_2(1)-\xi_1(1)\xi_1(1)-2\xi_2(1)=0$, that is $\xi_2=0.$
 \item We can use induction and conclude that $\xi_i=0,$ for all $i\ge 1.$
 \item Thus we conclude that  deformation is rigid, in case $A=k$.
 \end{enumerate}
\end{exmpl}
\begin{exmpl}
Let $k$ be a field. Take $A=k$. Let $M$  $N$ be a vector spaces over $k$. As in the previous example, $M$ and $N$ are modules over $A=k.$ Let $\xi_t$ and $\eta_t$ be deformations of $M$ and $N$, respectively. Then by using Example \ref{mode1}, $\xi_t=\xi_0$ and $\eta_t=\eta_0$. 
Let $\phi:M\to N$ be a module homomorphisms. Choose any $\phi_i\in Hom_k(M,N).$ Write $\phi_t=\sum_{i=0}^\infty \phi_it^i.$ We have 
$$\phi_t(\xi_t(r)m)=\sum_{i=0}^\infty \phi_i(\xi_0(r)m)t^i=\sum_{i=0}^\infty \phi_i(rm)t^i=\sum_{i=0}^\infty r\phi_i(m)t^i$$
and 
$$ \eta_t(r)\phi_t(m)=\eta_0(r)\sum_{i=0}^\infty \phi_i(m)t^i=\sum_{i=0}^\infty r\phi_i(m)t^i.$$
Thus $\phi_t(\xi_t(r)m)=\eta_t(r)\phi_t(m)$ and hence  $\phi_t=\sum_{i=0}^\infty \phi_it^i$ is a deformation of $\phi.$
\end{exmpl}

\end{document}